\documentclass[12 pt]{amsart}
\usepackage{amsmath,times,epsfig,amssymb,amsbsy,amscd,amsfonts,amstext,color,bm}
\usepackage[arrow,matrix]{xy}

\usepackage{tikz}

\numberwithin{equation}{section}

\theoremstyle{plain}
\newtheorem{theorem}{Theorem}[section]
\newtheorem{corollary}[theorem]{Corollary}
\newtheorem{lemma}[theorem]{Lemma}

\newtheorem{proposition}[theorem]{Proposition}

\newtheorem{problem}[theorem]{Problem}

\theoremstyle{definition}
\newtheorem{definition}[theorem]{Definition}
\newtheorem{example}[theorem]{Example}

\theoremstyle{remark}
\newtheorem{remark}[theorem]{Remark}

\newcommand{\A}{\mathcal{A}}
\newcommand{\Z}{\mathbb{Z}}

\newcommand{\R}{\mathbb{R}}
\newcommand{\C}{\mathbb{C}}
\newcommand{\CP}{\mathbb{CP}}
\newcommand{\RP}{\mathbb{RP}}
\newcommand{\scS}{\mathcal{S}}
\newcommand{\scL}{\mathcal{L}}

\newcommand{\rank}{\operatorname{rank}}
\newcommand{\codim}{\operatorname{codim}}

\newcommand{\Hom}{\operatorname{Hom}}
\newcommand{\Ker}{\operatorname{Ker}}
\renewcommand{\Im}{\operatorname{Im}}
\newcommand{\ID}{\mathcal{A_{ID}}}
\newcommand{\decID}{d\mathcal{A_{ID}}}

\begin{document}

\title[Double coverings]{Double coverings of arrangement complements and 
$2$-torsion in Milnor fiber homology}

\begin{abstract}
We prove that the mod $2$ Betti numbers of double coverings of 
a complex hyperplane arrangement complement are combinatorially 
determined. The proof is based on a relation between the mod $2$ Aomoto complex 
and the transfer long exact sequence. 

Applying the above result to the icosidodecahedral arrangement ($16$ 
planes in the three dimensional space related to the icosidodecahedron), 
we conclude that the first homology of the Milnor fiber 
has non-trivial $2$-torsion. 
\end{abstract}

\author{Masahiko Yoshinaga}
\address{Masahiko Yoshinaga, 
Department of Mathematics, Faculty of Science, Hokkaido University, 
Kita 10, Nishi 8, Kita-Ku, Sapporo 060-0810, Japan.}
\email{yoshinaga@math.sci.hokudai.ac.jp}

\thanks{The author thanks Professor Toshiyuki Akita for helpful discussion on mod $2$ Betti numbers. The author also deeply thanks the referee for careful reading of the 
manuscript and for many suggestions. 
This work was partially supported by 
JSPS KAKENHI Grant Numbers JP18H01115, JP15KK0144}

\dedicatory{Dedicated to the memory of Stefan Papadima}


\date{\today}
\maketitle

\tableofcontents


\section{Introduction}
\label{sec:intro}

An arrangement of hyperplanes is a finite collection of 
hyperplanes in a finite dimensional vector (affine, or projective) 
space. For a complex arrangement, we can associate 
several topological spaces: the complement, Milnor fiber, 
their covering spaces, boundary manifolds and so on. 
These provides important spaces in topology 
such as the classifying space of the pure braid (Artin) group 
and their subgroups. 

Another important aspect of arrangement is the combinatorial 
structure. An arrangement defines the poset of subspaces 
expressed as intersections of hyperplanes. These subspaces 
can be also considered as (the closure of) strata of a stratification 
of the total space. 
As is naturally expected, the poset of intersections has 
a lot of information on the associated topological spaces. 
Indeed, some of topological invariant (e.g., cohomology ring 
of the complement \cite{O-S}. See also \S \ref{sec:general} 
for more details) is determined by the poset of intersections, 
while some other can not be determined by the poset 
(e.g., the fundamental group \cite{ryb}). 
Furthermore, there are many invariant whose relations to 
the poset of intersections are yet unclear. 

Recently, Milnor fibers, and more generally, covering spaces 
of the arrangement complements received considerable amount 
of attention. 
See for \cite{suc-mil} for survey on the topic. 
Among other results, we just recall some of recent results on 
the Milnor fiber of an arrangement: 
\begin{itemize}
\item 
Explicit computations for interesting examples 
\cite{bdy, dim-sti, mac-pap, y-mil}. 
\item 
Upper/lower bounds of monodromy eigenspaces 
\cite{bai-set, bds, ps-sp, williams, y-mil}. 
\item 
Vanishing of non-trivial monodromy eigenspaces 
\cite{deg, salser-tw, salser-arr}. 
\item 
Examples of arrangements with multiplicities whose first homology group 
of the Milnor fiber has torsion \cite{cds}. 
\item 
Examples of arrangements whose higher degree homology of the Milnor fiber has 
torsion \cite{ds-multinets}. 
\item 
Purely combinatorial description of monodromy eigenspaces 
for line arrangements which have only double and triple intersections. 
\cite{dim-tri, lib, PS-modular}. 
\end{itemize}

The purpose of this paper is to study the mod $2$ homology and 
$2$-torsion in the integral homology of covering spaces of arrangement 
complements. The main results of this paper are as follows. 
The first result is concerning mod $2$ homology of the double 
covering. 
\begin{theorem}
(See Theorem \ref{thm:2cover} for more precise statement.) 
The mod $2$ Betti numbers of double covers of arrangement 
complements are combinatorially determined. 
\end{theorem}
More precisely, we will obtain a combinatorial expression of the 
rank of mod $2$ homology in terms of the mod $2$ Aomoto complex. 

The second result is about the Milnor fiber of the icosidodecahedral 
arrangement which is the arrangement of $16$ planes 
in $\R^3$ associated to the icosidodecahedron. 
(See \S \ref{sec:icosido} for definition). 
\begin{theorem}
(See Theorem \ref{milnortorsion} (3).) 
Let $F_{\ID}$ be the Milnor fiber of the icosidodecahedral arrangement 
$\ID$. 
Then $H_1(F_{\ID}, \Z)$ has $2$-torsion. 
\end{theorem}

The organization of this paper is as follows. 
In \S \ref{sec:general}, we review several results concerning 
Milnor fiber of arrangements, intersection poset, Orlik-Solomon algebra, 
and Aomoto complex, 
which are necessary for stating results and proofs.

\S \ref{sec:fincovercomb} summarizes results on abelian covering spaces. 
After recalling the transfer long exact sequence of a double covering in 
\S \ref{sec:finabcov}-\S \ref{sec:transfer}, we prove, in \S \ref{sec:liftable}, 
a key lemma (Lemma \ref{lem:liftable}) 
which guarantees the first mod $2$ Betti number does not decrease 
by taking certain (``$\Z_4$-liftable'') double coverings. 
In \S \ref{sec:comb}, we prove the first main result, that is 
a combinatorial formula for the mod $2$ Betti numbers of double coverings 
of an arrangement complement (Theorem \ref{thm:2cover}). 

In \S \ref{sec:icosido}, we introduce the icosidodecahedral arrangement 
$\ID$, and prove in \S \ref{sec:mod2icosi} and \S \ref{sec:2toricosi} 
that the first homology of the Milnor fiber has $2$-torsion. 

\medskip 

\noindent
{\bf Notation and Convention.} We will use the following convention 
throughout the paper. 
\begin{itemize}
\item 
In this paper, $\Z_n$ denotes the cyclic group 
(also a finite ring) $\Z/n\Z$. 
\item 
For a space $X$ with the homotopy type of a finite CW complex, 
$b_i(X)=\rank_{\Z}H^i(X, \Z)$ denotes the Betti number, and 
$\overline{b_i}(X)=\rank_{\Z_2}H^i(X, \Z_2)$ denotes the mod $2$ Betti number. 
\item 
A \emph{covering} $Y\longrightarrow X$ always means an unbranched covering. 
Unless otherwise stated, we assume $X$ and $Y$ are connected. 
\item 
We will assume that the base point $x_0\in X$ is fixed 
when we consider the fundamental group $\pi_1(X)=\pi_1(X, x_0)$. 
\item 
For an element $\gamma\in\pi_1(X)$, denote by $[\gamma]$ its 
homology class. If $(Y, y_0)\longrightarrow (X, x_0)$ is a covering, 
$\tilde{\gamma}$ denotes the lifting of $\gamma$ 
starting from the the base point $y_0$. Note that 
$\tilde{\gamma}$ is not necessarily a closed curve. 
\end{itemize}

\section{Generalities on hyperplane arrangements}
\label{sec:general}

In this section, we summarize notation and several results 
on hyperplane arrangements. See \cite{ot, ot-int} for more details. 

Let $\A=\{H_1, H_2, \dots, H_n\}$ be a collection of affine hyperplanes 
in $\C^\ell$. We denote by $M(\A)=\C^\ell\smallsetminus\bigcup_{i=1}^nH_i$ 
the complement. 
Let $\overline{H_\infty}$ be a projective hyperplane in 
$\CP^\ell$. 
Using the identification $\CP^\ell=\C^\ell\sqcup \overline{H_\infty}$, 
$\A$ naturally induces an arrangement of $n+1$ projective hyperplanes 
$\overline{\A}=\{\overline{H_1}, \dots, \overline{H_n}, \overline{H_\infty}\}$. 
A projective hyperplane $\overline{H_i}$ defines a linear hyperplane 
$\widetilde{H_i}$ in $\C^{\ell+1}$. The collection of these linear hyperplanes 
$\widetilde{\A}=\{\widetilde{H_1}, \dots, \widetilde{H_n}, \widetilde{H_\infty}\}$ 
is called the coning of $\A$ (which is also denoted by $c\A$). 
We can also define the opposite operation, which is called the deconing 
of $\widetilde{\A}$ with respect to $\widetilde{H_\infty}$ and denoted by 
$\A=d_{\widetilde{H_\infty}}\widetilde{\A}$. 

There is a natural projection $p:M(c\A)\longrightarrow M(\A)$. It is easily seen 
that $M(c\A)\simeq M(\A)\times\C^\times$. 
Let $\alpha_i:\C^{\ell+1}\longrightarrow\C$ be a non-zero 
linear form such that $\widetilde{H_i}=\alpha_i^{-1}(0)$ ($i=1, \dots, n, \infty$). 
Then $Q=\prod_i\alpha_i$ is a homogeneous polynomial of degree $n+1$. 
$F:=\{x\in\C^{\ell+1}\mid Q(x)=1\}\subset\C^{\ell+1}$ is called 
the Milnor fiber of $c\A$. Since $Q$ is homogeneous, the cyclic group 
$\Z/(n+1)\Z\simeq\{\zeta\in\C\mid \zeta^{n+1}=1\}$ acts on $F$ defined 
by the map $\mu:F\longrightarrow F, x\longmapsto \zeta\cdot x$. 
This action is nothing but the monodromy 
action of the fibration $Q:M(c\A)\longrightarrow\C^\times$. 
The monodromy map $\mu:F\longrightarrow F$ induces a linear 
map on homology $\mu_*:H_k(F, \Z)\longrightarrow H_k(F, \Z)$. 
Since the map has finite order, the homology with coefficients in $\C$ 
is decomposed into the direct sum of eigenspaces, 
\[
H_k(F, \C)=\bigoplus_{\zeta^{n+1}=1}H_k(F, \C)_\zeta, 
\]
where $H_k(F, \C)_\zeta$ is the $\zeta$-eigenspace of $\mu_*$. 
Since $M(\A)$ can be identified with the quotient $F/\langle\mu\rangle$, 
we have $H_k(F, \C)_1\simeq H_k(M(\A), \C)$. 

Given an affine arrangement $\A=\{H_1, \dots, H_n\}$, 
non-empty intersections $X=\bigcap_{H\in\mathcal{S}}H$ 
($\mathcal{S}\subset\A$) form a poset with respect to reverse inclusion, 
which is denoted by $L(\A)$ and called the intersection poset. 
We say that $\mathcal{S}\subset\A$ does not intersect if 
$\bigcap_{H\in\mathcal{S}}H =\emptyset$. 
A subset $\mathcal{S}\subset\A$ is called dependent if 
$\bigcap_{H\in\mathcal{S}}H\neq\emptyset$ and 
$\codim\bigcap_{H\in\mathcal{S}}H < \#\mathcal{S}$. 
For $X\in L(\A)$, we denote by $\A_X:=\{H\in\A\mid H\supset X\}$ the 
localization of $\A$ at $X$. 

From the intersection poset, the Orlik-Solomon algebra 
$A_{\Z}^\bullet(\A)$ is defined as follows. 
Let $E=\bigoplus_{i=1}^n\Z e_i$ be 
the free abelian group generated by the symbols $e_i$ corresponding to 
the hyperplanes $H_i$. Let 
$\wedge E=\Z\oplus E\oplus E^{\wedge 2}\oplus\cdots\oplus E^{\wedge n}$ 
be the exterior algebra on $E$. 
For given $\mathcal{S}=\{i_1, i_2, \dots, i_k\}\subset\A$ with 
$i_1<\dots<i_k$, let $e_{\mathcal{S}}:=e_{i_1}\wedge\cdots\wedge e_{i_k}$, 
and define $\partial e_{\mathcal{S}}\in E^{\wedge (k-1)}$ by 
\[
\partial e_{\mathcal{S}}=\sum_{p=1}^k(-1)^{p-1}e_{i_1}\wedge\dots
\widehat{e_{i_p}}\wedge\dots\wedge e_{i_k}. 
\]
Define $I(\A)$ to be the ideal of $\wedge E$ generated by the following 
elements. 
\[
\{\partial e_{\mathcal{S}}\mid \mathcal{S}\mbox{ is dependent}\}
\cup
\{e_{\mathcal{S}}\mid\mathcal{S}\mbox{ does not intersect}\}. 
\]
The quotient ring $A_{\Z}^*(\A):=\wedge E/I(\A)$ is called the 
Orlik-Solomon algebra. For any abelian group $R$, denote 
$A_R^*(\A):=A_{\Z}^*(\A)\otimes_{\Z}R$. Note that when $R$ is a 
commutative ring, $A_R^*(\A)$ becomes a graded commutative $R$-algebra. 
Brieskorn's Lemma \cite[Corollary 3.73]{ot} shows that the degree $k$ 
component of the Orlik-Solomon algebra is 
\[
A_R^k(\A)\simeq\bigoplus_{\substack{X\in L(\A)\\ \codim X=k}}A_R^k(\A_X). 
\]

Orlik and Solomon \cite{O-S} proved that each homology group 
$H_k(M(\A), \Z)$ is torsion free and the cohomology ring is isomorphic 
to the Orlik-Solomon algebra. Namely, for any commutative ring $R$, 
we have an isormorphism 
\[
H^*(M(\A), R)\simeq A_R^*(\A) 
\]
as graded commutative $R$-algebras. 

The following will be used later. 
\begin{lemma}
\label{lem:lines}
Let $\A=\{H_1, \dots, H_n\}$ be an arrangement of $n$ lines ($n\geq 3$) 
in $\C^2$ with unique intersection (Figure \ref{fig:lines}). 
Let $R$ be an integral domain. 
Let $\eta=\sum_{i=1}^n a_ie_i$ and 
$\omega=\sum_{i=1}^n b_ie_i \in A_R^1(\A)$. 
Then the following are equivalent. 
\begin{itemize}
\item[(a)] $\eta\wedge\omega=0$. 
\item[(b)] $\sum_{i=1}^na_i=\sum_{i=1}^nb_i=0$ or 
$\omega, \eta$ are linearly dependent 
(i. e., there are $c_1, c_2\in R$, not both zero, such that 
$c_1\eta+c_2\omega=0$). 
\end{itemize}
\end{lemma}
\begin{proof}
See \cite[Proposition 2.1]{yuz-bos} (or \cite[Lemma 3.1]{to-yo}).
\end{proof}

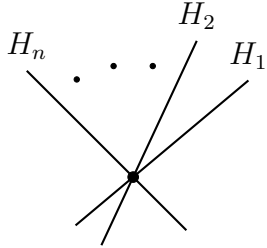
\begin{figure}[htbp]
\centering
\begin{tikzpicture}[scale=1]

\draw[thick] (0,0)++(220:1)--+(40:3) node[above]{$H_1$};
\draw[thick] (0,0)++(245:1)--+(65:3) node[above]{$H_2$};
\draw[thick] (0,0)++(315:1)--+(135:3) node[above]{$H_n$};

\filldraw[fill=black, draw=black] (0,0) circle(2pt);

\filldraw[fill=black, draw=black] (80:1.5) circle(1pt);
\filldraw[fill=black, draw=black] (100:1.5) circle(1pt);
\filldraw[fill=black, draw=black] (120:1.5) circle(1pt);

\end{tikzpicture}
\caption{$n$ lines with one intersection ($n\geq 3$)}
\label{fig:lines}
\end{figure}

For a given $\omega\in A_R^1(\A)$, Orlik-Solomon algebra 
determines the cochain complex 
$(\A_R^\bullet(\A), \omega\wedge-)$, which is called the Aomoto 
complex. This complex plays crucial role in the computation of 
twisted cohomology groups \cite{ao-ki, esv, fal-coh, lib-yuz, ps-sp, 
stv, to-yo, yuz-bos}.

\section{Finite coverings and combinatorial structures}

\label{sec:fincovercomb}

\subsection{Finite abelian covering}

\label{sec:finabcov}

Let $G$ be a finite abelian group. 
Recall that 
a group homomorphism $w:\pi_1(X)\longrightarrow G$ 
determines a finite abelian covering $p(w):X^w\longrightarrow X$. 
Since $G$ is abelian, we have 
\[
\Hom(\pi_1(X), G)=
\Hom(H_1(X, \Z), G)\simeq
H^1(X, G). 
\]
The element $w\in H^1(X, G)$ is called the characteristic class of the covering. 
Note that 
$\Ker(w:\pi_1(X)\to G)$ is isomorphic to $\pi_1(X^w)$ 
and $\pi_1(X)/\pi_1(X^w)\simeq \Im(w:\pi_1(X)\to G)$. 

\subsection{Double covering and transfer exact sequence}
\label{sec:transfer}

(See \cite[\S 4.3]{hau} for details.) 
Now we consider the double covering 
$p:Y\longrightarrow X$ with the unique nontrivial deck transformation 
$\sigma:Y\longrightarrow Y$. Assume that $Y$ is connected. 
Denote the characteristic 
class by $w\in H^1(X, \Z_2)=\Hom(\pi_1(X), \Z_2)$. 
Recall that the transfer map 
\[
\tau_*:H_k(X, \Z_2)\longrightarrow H_k(Y, \Z_2),
\]
is given by $[\gamma]\longmapsto[p^{-1}(\gamma)]$. 
Denote by 
$\tau^*:H^k(Y, \Z_2)\longrightarrow H^k(X, \Z_2)$ the induced 
map between cohomology groups. 

\begin{example}
\label{ex:lift}
Let us closely look at the transfer map $\tau_*$ in degree one. 
Let $\gamma\in\pi_1(X)$. Then either 
$w(\gamma)=0$ or $\neq 0$. 
\begin{itemize}
\item[(1)] 
Suppose $w(\gamma)\neq 0$, equivalently, $\gamma\notin\pi_1(Y)
(\subset\pi_1(X))$. Then the lift $\widetilde{\gamma}$ 
is no longer a cycle. Note that $p^{-1}(\gamma)$ is the union of 
$\widetilde{\gamma}$ and $\sigma(\widetilde{\gamma})$, which 
is the lift of $\gamma^2$. 
Since $\gamma^2\in\pi_1(Y)\subset\pi_1(X)$, 
we have $\tau_*([\gamma])=[\gamma^2]$. 
\item[(2)] 
Suppose $w(\gamma)= 0$. Then the lift $\widetilde{\gamma}$ is 
a cycle on $Y$. Hence 
$\tau_*([\gamma])=[\widetilde{\gamma}]+[\sigma(\widetilde{\gamma})]$. 
\end{itemize}
\end{example}
The transfer map fits into the following long exact sequence. 
\begin{equation}
\cdots
H^{k-1}(X, \Z_2)
\stackrel{w\cup}{\longrightarrow} 
H^{k}(X, \Z_2)
\stackrel{p^*}{\longrightarrow} 
H^k(Y, \Z_2)
\stackrel{\tau^*}{\longrightarrow} 
H^{k}(X, \Z_2)
\stackrel{w\cup}{\longrightarrow} 
H^{k+1}(X, \Z_2)
\cdots. 
\end{equation}

\subsection{Mod $2$ Betti number of $\Z_4$-liftable double coverings}

\label{sec:liftable}

In this section, we prove an inequality between 
the mod $2$ first Betti numbers of a double covering $Y\longrightarrow X$, 
which will play a crucial role later in \S \ref{sec:2toricosi}. 

\begin{lemma}
\label{lem:liftable}
Let $w:\pi_1(X)\longrightarrow\Z_4$ be a surjective homomorphism. 
By composing the canonical surjective homomorphism $\Z_4\rightarrow\Z_2$, 
we obtain an epimorphism $\overline{w}:\pi_1(X)\longrightarrow\Z_2$. 
Consider the associated double coverings 
$X^w\longrightarrow X^{\overline{w}}$ 
and 
$X^{\overline{w}}\longrightarrow X$. 
\begin{enumerate}
\item 
$\overline{w}\in H^1(X, \Z_2)$ satisfies $\overline{w}\cup\overline{w}=0$. 
\item 
$\overline{b_1}(X^{\overline{w}})\geq\overline{b_1}(X)$. 
\end{enumerate}
\end{lemma}
\begin{proof}
(1) 
Recall that the exact sequence of abelian groups $0\to\Z_2\to\Z_4\to\Z_2\to 0$ 
induces the exact sequence 
\[
H^1(X, \Z_4)\stackrel{\varphi}{\to} 
H^1(X, \Z_2)\stackrel{\beta}{\to} H^2(X, \Z_2). 
\]
The first map $\varphi$ 
sends $w$ to $\overline{w}$. The second map 
$\beta: H^1(X, \Z_2)\to H^2(X, \Z_2)$ is the so-called 
Bockstein homomorphism, which is $\beta(x)=x\cup x$ \cite[Section 4.L]{hatcher}. 
Since $\beta\circ\varphi=0$, we have $\overline{w}\cup\overline{w}=0$. 

(2) Since $p^*:H^0(X, \Z_2)\simeq H^0(X^{\overline{w}}, \Z_2)$ is 
an isomorphism, 
the beginning of the transfer exact sequence is as follows. 
\[
0\to
H^{0}(X, \Z_2)
\stackrel{\overline{w}\cup}{\to} 
H^{1}(X, \Z_2)
\stackrel{p^*}{\to} 
H^1(X^{\overline{w}}, \Z_2)
\stackrel{\tau^*}{\to} 
H^{1}(X, \Z_2)
\stackrel{\overline{w}\cup}{\to} 
H^{2}(X, \Z_2). 
\] 
We have 
\[
\rank_{\Z_2}H^1(X^{\overline{w}}, \Z_2)=
\rank_{\Z_2}H^1(X, \Z_2)-1+\rank_{\Z_2}\Im(\tau^*). 
\]
Since $\overline{w}\in\Ker(\overline{w}\cup)=\Im(\tau^*)$, 
$\rank_{\Z_2}\Im(\tau^*)\geq 1$. 
This completes the proof. 
\end{proof}

\begin{remark}
Without the assumption of $\Z_4$-liftability in Lemma \ref{lem:liftable} (2), 
the inequality between mod $2$ Betti numbers does not hold in general. 
For example, 
the double covering $S^m\longrightarrow\RP^m$ does not satisfy the 
inequality for $m\geq 2$. Indeed $\rank_{\Z_2}H^1(S^m, \Z_2)=0$ 
and $\rank_{\Z_2}H^1(\RP^m, \Z_2)=1$. 
\end{remark}

\begin{corollary}
\label{cor:increase}
\begin{enumerate}
\item 
Let $X_k\to X_{k-1}\to\cdots\to X_0$ ($k\geq 2$) 
be a tower of double coverings of connected spaces 
(i.e., each $X_i\to X_{i-1}$ is a double covering) 
such that $X_k\to X_0$ is a cyclic $\Z_{2^k}$-covering. 
Then 
\[
\overline{b_1}(X_{k-1})\geq 
\overline{b_1}(X_{k-2})\geq 
\cdots\geq 
\overline{b_1}(X_{1})\geq 
\overline{b_1}(X_0).  
\]
\item 
Let $w:\pi_1(X)\longrightarrow \Z$ be a surjective 
homomorphism, and $\overline{w}:\pi_1(X)\longrightarrow \Z_{2^{k}}$ be 
the induced surjection ($k\geq 1$). Then 
\[
\overline{b_1}(X^{\overline{w}})\geq\overline{b_1}(X). 
\]
\end{enumerate}
\end{corollary}
\begin{proof}
(1) is proved by induction on $k$ using Lemma \ref{lem:liftable}. 
(2) follows immediately from (1). 
\end{proof}

\subsection{Double coverings of arrangement complements}
\label{sec:comb}

Now let us formulate a problem asking whether 
the Betti numbers of finite coverings of arrangement complements 
are combinatorially determined. 

\begin{problem}
\label{pbm:comb}
Let $\A$ be an arrangement in $\C^\ell$. 
Let $G$ be a finite abelian group, and $R$ be an abelian group. 
Let $w\in A_G^1(\A)\simeq\Hom(\pi_1(M(\A)), G)$. 
Describe the cohomology groups $H^k(M(\A)^w, R)$ (or their ranks) 
in terms of $L(\A)$ and $w\in A_G^1(\A)$. 
\end{problem}

\begin{example}
Let $\A=\{H_1, \dots, H_n\}$ be an arrangement of $n$ hyperplanes and 
$G=\Z_{n+1}$ be the cyclic group of order $(n+1)$. 
Let $w:=e_1+e_2+\dots+e_n\in A_{\Z_{n+1}}^1(\A)$. Then 
the corresponding covering $M(\A)^w$ is homeomorphic to 
the Milnor fiber $F$ of the coning $c\A$. 
\end{example}

Problem \ref{pbm:comb} is widely open. 
Indeed, many research problems are 
related to Problem \ref{pbm:comb} \cite{cs-char, hironaka, lib-cov, sak}. 
For example, the notion of 
characteristic variety is deeply related to the computation of 
$\rank_{\Z}H^k(M(\A)^w, \Z)$. 
See \cite{suc-fun, suc-mil} for surveys of the topic. 

As a special case of Problem \ref{pbm:comb}, 
we now prove that the mod $2$ Betti numbers of the double covering 
$M(\A)^w$ of an arrangement complement are combinatorially determined. 
\begin{theorem}
\label{thm:2cover}
Let $w\in A_{\Z_2}^1(\A), w\neq 0$. Then the $k$-th mod $2$ Betti number 
($k\geq 0$) of the double covering $M(\A)^w$ is expressed as follows. 
\begin{equation}
\label{eq:mod2Bettiformula}
\overline{b_k}(M(\A)^w)
=b_k(M(\A))+\rank_{\Z_2} H^k(A_{\Z_2}^\bullet(\A), w\wedge-).  
\end{equation}
\end{theorem}

\begin{proof}
Denote 
$\mathcal{Z}^k:=\Ker(w\wedge: A_{\Z_2}^{k}(\A)\to A_{\Z_2}^{k+1}(\A))$ 
and 
$\mathcal{B}^k:=\Im(w\wedge: A_{\Z_2}^{k-1}(\A)\to A_{\Z_2}^{k}(\A))$. 
From the transfer long exact sequence, we have the following 
exact sequence. 
\[
0
\longrightarrow
\mathcal{B}^k
\longrightarrow
A_{\Z_2}^k(\A)
\longrightarrow
H^k(M(\A)^w, \Z_2)
\longrightarrow
\mathcal{Z}^k
\longrightarrow
0. 
\]
Since $\mathcal{Z}^k/\mathcal{B}^k\simeq H^k(A_{\Z_2}^\bullet(\A), w\wedge-)$, 
\[
\begin{split}
\rank_{\Z_2} H^k(M(\A)^w, \Z_2)
&=
\rank_{\Z_2} A_{\Z_2}^k(\A)+\rank_{\Z_2}\mathcal{Z}^k - \rank_{\Z_2} \mathcal{B}^k\\
&=
b_k(M(\A))+\rank_{\Z_2} H^k(A_{\Z_2}^\bullet(\A), w\wedge-). 
\end{split}
\]
\end{proof}
\begin{remark}
Note that Theorem \ref{thm:2cover} 
gives only mod $2$ Betti numbers. It is not clear whether we can 
combinatorially describe Betti numbers or (mod $2$) cohomology ring structure 
of the double covering. 
\end{remark}

\section{$2$-torsions in Milnor fiber homology}

\subsection{Icosidodecahedral arrangement $\ID$}
\label{sec:icosido}

\begin{definition}
The icosidodecahedral arrangement $\ID$ is the coning of the 
$15$ affine lines in Figure \ref{fig:triplestar}. Namely, $\ID$ consists of 
$16$ planes in $\R^3$. 
(Another deconing of $\ID$ is depicted in Figure \ref{fig:triangle}.) 
\end{definition}

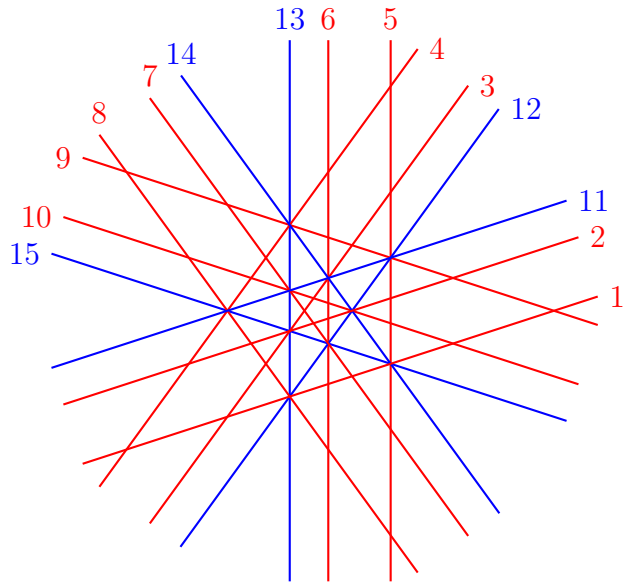
\begin{figure}[htbp]
\centering
\begin{tikzpicture}[scale=1.2]




\draw[thick, red] (0.809,0)++(0,-3) -- +(90:6) node[above, red]{$5$};
\draw[thick, blue] (-0.309,0)++(0,-3) -- +(90:6) node[above, blue]{$13$};
\draw[thick, red] (0.118,0)++(0,-3) -- +(90:6) node[above, red]{$6$};

\draw[thick, red] (72:0.809)++(162:3)  node[left, red]{$9$}-- +(342:6);
\draw[thick, blue] (252:0.309)++(162:3)  node[left, blue]{$15$}-- +(342:6);
\draw[thick, red] (72:0.118)++(162:3)  node[left, red]{$10$}-- +(342:6);

\draw[thick, red] (144:0.809)++(234:3) -- +(54:6) node[right, red]{$4$};
\draw[thick, blue] (322:0.309)++(234:3) -- +(54:6) node[right, blue]{$12$};
\draw[thick, red] (144:0.118)++(234:3) -- +(54:6) node[right, red]{$3$};

\draw[thick, red] (216:0.809)++(126:3)  node[above, red]{$8$}-- +(306:6);
\draw[thick, blue] (36:0.309)++(126:3)  node[above, blue]{$14$}-- +(306:6);
\draw[thick, red] (216:0.118)++(126:3)  node[above, red]{$7$}-- +(306:6);

\draw[thick, red] (288:0.809)++(198:3) -- +(18:6) node[right, red]{$1$};
\draw[thick, blue] (108:0.309)++(198:3) -- +(18:6) node[right, blue]{$11$};
\draw[thick, red] (288:0.118)++(198:3) -- +(18:6) node[right, red]{$2$};

\end{tikzpicture}
\caption{A deconing $d_{H_{16}}\ID$ 
of the icosidodecahedral arrangement (with respect 
to a blue line at infinity ${\color{blue} H_{16}}$). The coloring expresses 
the nontrivial cocycle in the mod $2$ Aomoto complex (\S \ref{sec:mod2icosi}).} 
\label{fig:triplestar}
\end{figure}
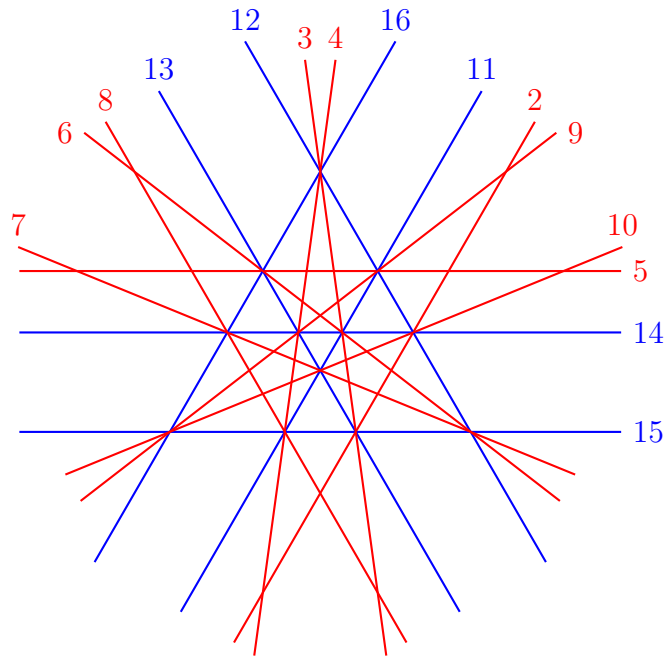
\begin{figure}[htbp]
\centering
\begin{tikzpicture}[scale=1]

\draw[thick, blue] (0,0)++(180:2)--+(8,0) node[right, blue]{$15$};
\draw[thick, blue] (0,0)++(180:2)++(0,1.323169)--+(8,0) node[right, blue]{$14$};
\draw[thick, red] (0,0)++(180:2)++(0,2.140325)--+(8,0) node[right, red]{$5$};

\draw[thick, blue] (240:2)--+(60:8)  node[above, blue]{$16$};
\draw[thick, blue] (330:1.3231)++(240:2)--+(60:8)  node[above, blue]{$11$};
\draw[thick, red] (330:2.14)++(240:2)--+(60:8) node[above, red]{$2$};

\draw[thick, blue] (4,0)++(300:2)--+(120:8) node[above, blue]{$12$};
\draw[thick, blue] (210:1.3231)++(4,0)++(300:2)--+(120:8) node[above, blue]{$13$};
\draw[thick, red] (210:2.14)++(4,0)++(300:2)--+(120:8) node[above, red]{$8$};

\draw[thick, red] (202.23875:1.5)--+(22.23875:8) node[above, red]{$10$};
\draw[thick, red] (217.76124:1.5)--+(37.76124:8) node[right, red]{$9$};

\draw[thick, red] (4,0)++(337.761244:1.5)--+(157.761244:8) node[above, red]{$7$};
\draw[thick, red] (4,0)++(322.238756:1.5)--+(142.238756:8) node[left, red]{$6$};

\draw[thick, red] (60:4)++(97.761244:1.5) node[above, red]{$3$}--+(277.761244:8);
\draw[thick, red] (60:4)++(82.238756:1.5) node[above, red]{$4$}--+(262.238756:8);

\end{tikzpicture}
\caption{Another deconing $d_{H_1}\ID$ 
of the icosidodecahedral arrangement (with respect to a red line 
{\color{red} $H_1$} at infinity)}
\label{fig:triangle}
\end{figure}

Let us briefly comment on the naming ``Icosidodecahedral arrangement''. 
Actually, $\ID$ 
can be constructed from the icosidodecahedron as follows, 
which seems to be the most symmetric realization of $\ID$. 

The icosidodecahedron (Figure \ref{fig:icosidodeca}) 
is a polyhedron which is commonly 
obtained as vertex truncations of the icosahedron and the dodecahedron 
(by truncating mid points of edges). 
An icosidodecahedron has $32$ faces ($20$ triangles and $12$ pentagons), 
$60$ edges and $30$ vertices (Figure \ref{fig:icosidodeca}). 
We can choose $10$ edges to form the equator of the polyhedron, 
which are lying on a plane in $\R^3$. Similarly, 
we obtain $6$ planes in total (they correspond to the blue lines 
$H_{11}, \dots, H_{15}, H_{16}$ in 
Figures \ref{fig:triplestar} and \ref{fig:triangle}).  

Each pentagonal face of the icosidodecahedron 
has five diagonals. There are 
$60$ such diagonals in all. If we choose appropriately consecutive six of them, 
they are lying on a plane (the red diagonals in Figure \ref{fig:icosidodeca}). 
We obtain in this manner $10$ planes consisting of diagonals (Figure \ref{fig:idarr}). 
As a result, we have an arrangement of $16$ planes in $\R^3$, which is 
isomorphic to $\ID$. 

For computations, 
we use mainly the deconing from Figure \ref{fig:triplestar}. 

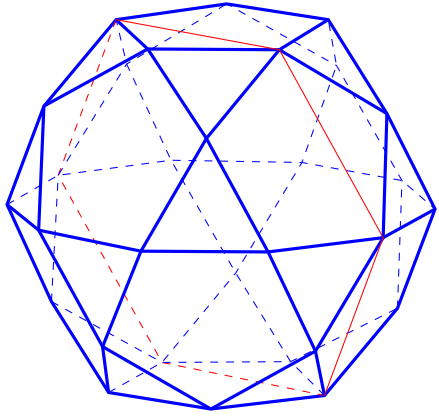
\begin{figure}[htbp]
\centering
\begin{tikzpicture}[scale=1]


\coordinate (a01) at (0.87,1.85);
\coordinate (a02) at (1.63,0.64);
\coordinate (a03) at (2.99,0.42);
\coordinate (a04) at (4.5,0.6);
\coordinate (a05) at (5.47,1.76);
\coordinate (a06) at (5.97,3.08);
\coordinate (a07) at (5.33,4.35);
\coordinate (a08) at (4.55,5.6);
\coordinate (a09) at (3.2,5.81);
\coordinate (a10) at (1.73,5.6);
\coordinate (a11) at (0.77,4.45);
\coordinate (a12) at (0.28,3.14);
\coordinate (a13) at (2.34,1.04);
\coordinate (a14) at (4.05,1.05);
\coordinate (a15) at (3.32,2.23);
\coordinate (a16) at (5.53,3.46);
\coordinate (a17) at (4.2,3.7);
\coordinate (a18) at (4.66,4.99);
\coordinate (a19) at (1.87,5.01);
\coordinate (a20) at (2.46,3.73);
\coordinate (a21) at (0.96,3.53);
\coordinate (a22) at (2.15,5.21);
\coordinate (a23) at (3.9,5.2);
\coordinate (a24) at (2.93,4.02);
\coordinate (a25) at (3.75,2.51);
\coordinate (a26) at (2.06,2.52);
\coordinate (a27) at (0.7,2.8);
\coordinate (a28) at (1.55,1.25);
\coordinate (a29) at (4.38,1.19);
\coordinate (a30) at (5.28,2.7);

\draw[very thick, blue] (a01)--(a02)--(a03)--(a04)--(a05)--(a06)--(a07)--(a08)--(a09)--(a10)--(a11)--(a12)--cycle;

\draw [dashed, thin, red] (a10)--(a21)--(a13)--(a04);

\draw[very thick, blue] (a12)--(a27)--(a26)--(a25)--(a30)--(a06);
\draw[very thick, blue] (a02)--(a28)--(a26)--(a24)--(a23)--(a08);
\draw[very thick, blue] (a03)--(a28)--(a27)--(a11);
\draw[very thick, blue] (a03)--(a29)--(a30)--(a07);
\draw[very thick, blue] (a04)--(a29)--(a25)--(a24)--(a22)--(a10);
\draw[very thick, blue] (a07)--(a23)--(a22)--(a11);

\draw [dashed, thin, blue] (a12)--(a21)--(a20)--(a17)--(a16)--(a06);
\draw [dashed, thin, blue] (a01)--(a21)--(a19)--(a09);
\draw [dashed, thin, blue] (a01)--(a13)--(a14)--(a05);
\draw [dashed, thin, blue] (a02)--(a13)--(a15)--(a17)--(a18)--(a08);
\draw [dashed, thin, blue] (a04)--(a14)--(a15)--(a20)--(a19)--(a10);
\draw [dashed, thin, blue] (a05)--(a16)--(a18)--(a09);


\draw [red] (a04)--(a30)--(a23)--(a10);



\end{tikzpicture}
\caption{Icosidodecahedron (blue) and a diagonal plane (red)}
\label{fig:icosidodeca}
\end{figure}

\begin{figure}[htbp]
\centering
\begin{tikzpicture}[scale=1]


\coordinate (a01) at (0.87,1.85);
\coordinate (a02) at (1.63,0.64);
\coordinate (a03) at (2.99,0.42);
\coordinate (a04) at (4.5,0.6);
\coordinate (a05) at (5.47,1.76);
\coordinate (a06) at (5.97,3.08);
\coordinate (a07) at (5.33,4.35);
\coordinate (a08) at (4.55,5.6);
\coordinate (a09) at (3.2,5.81);
\coordinate (a10) at (1.73,5.6);
\coordinate (a11) at (0.77,4.45);
\coordinate (a12) at (0.28,3.14);
\coordinate (a13) at (2.34,1.04);
\coordinate (a14) at (4.05,1.05);
\coordinate (a15) at (3.32,2.23);
\coordinate (a16) at (5.53,3.46);
\coordinate (a17) at (4.2,3.7);
\coordinate (a18) at (4.66,4.99);
\coordinate (a19) at (1.87,5.01);
\coordinate (a20) at (2.46,3.73);
\coordinate (a21) at (0.96,3.53);
\coordinate (a22) at (2.15,5.21);
\coordinate (a23) at (3.9,5.2);
\coordinate (a24) at (2.93,4.02);
\coordinate (a25) at (3.75,2.51);
\coordinate (a26) at (2.06,2.52);
\coordinate (a27) at (0.7,2.8);
\coordinate (a28) at (1.55,1.25);
\coordinate (a29) at (4.38,1.19);
\coordinate (a30) at (5.28,2.7);
\coordinate (c) at (3.12,3.12);

\draw[very thick, blue] (a01)--(a02)--(a03)--(a04)--(a05)--(a06)--(a07)--(a08)--(a09)--(a10)--(a11)--(a12)--cycle;

\draw[very thick, blue] (a12)--(a27)--(a26)--(a25)--(a30)--(a06);
\draw[very thick, blue] (a02)--(a28)--(a26)--(a24)--(a23)--(a08);
\draw[very thick, blue] (a03)--(a28)--(a27)--(a11);
\draw[very thick, blue] (a03)--(a29)--(a30)--(a07);
\draw[very thick, blue] (a04)--(a29)--(a25)--(a24)--(a22)--(a10);
\draw[very thick, blue] (a07)--(a23)--(a22)--(a11);


\draw [thick, red] (a01)--(a27)--(a24)--(a07);
\draw [thick, orange] (a01)--(a28)--(a25)--(a07);
\draw [thick, green] (a12)--(a28)--(a29)--(a06);
\draw [red] (a02)--(a12);
\draw [thick, black] (a02)--(a27)--(a22)--(a08);
\draw [thick, purple] (a03)--(a26)--(a22)--(a09);
\draw [thick, pink!30!black] (a03)--(a25)--(a23)--(a09);
\draw [thick, olive] (a04)--(a30)--(a23)--(a10);
\draw [red] (a04)--(a06);
\draw [thick, gray] (a05)--(a29)--(a26)--(a11);
\draw [thick, green!50!black] (a05)--(a30)--(a24)--(a11);
\draw [red] (a08)--(a10);

\end{tikzpicture}
\caption{Icosidodecahedral arrangement}
\label{fig:idarr}
\end{figure}
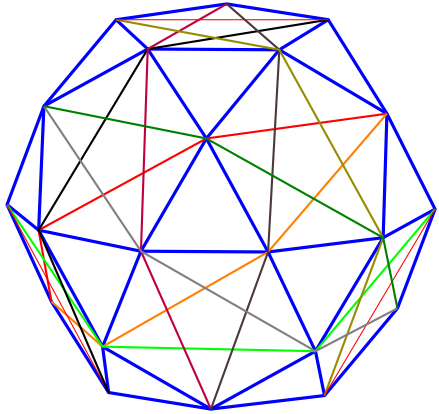

\subsection{Mod $2$ Aomoto complex of icosidodecahedral arrangement}

\label{sec:mod2icosi}

Let $\decID :=d_{H_{16}}\ID$ be the affine arrangement in Figure \ref{fig:triplestar}. 
Let $w_2:=e_1+\cdots+e_{15}\in A_{\Z_2}^1(\decID)$. For a subset 
$\scS\subset\decID$, let $e_{\scS}:=\sum_{H_i\in\scS}e_i$. Obviously, 
every element in $A_{\Z_2}^1(\decID)$ can be expressed as 
$e_{\scS}$, where $\scS$ is a subset of $\decID$. 

\begin{proposition}
\label{mod2aomoto}
$\rank_{\Z_2}H^1(A_{\Z_2}^\bullet(\decID), w_2\wedge-)=1$. 
\end{proposition}
\begin{proof}
Let $\scS_0:=\{1, 2, \dots, 10\}$ (red in Figure \ref{fig:triplestar}) 
 and $\scS_1:=\{11, 12, 13, 14, 15\}$ (blue in Figure \ref{fig:triplestar}). 
Note that at each intersection $p$, the localization $(\decID)_p$ consists of 
either 
\begin{itemize}
\item
two lines from $\scS_0$, or 
\item 
two from $\scS_0$ and two from $\scS_1$. 
\end{itemize}
This property, together with Lemma \ref{lem:lines}, 
enables us to conclude 
\[
w_2\wedge e_{\scS_0}=w_2\wedge e_{\scS_1}=0. 
\]
Therefore, $\rank_{\Z_2}H^1(A_{\Z_2}^\bullet(\decID), w_2\wedge-)\geq 1$. 

Next we show that $[e_{\scS_0}]=[e_{\scS_1}]$ is the unique nonzero 
cohomology class. 
Suppose $w_2\wedge e_{\scS}=0$ for some $\scS\subset\decID$. 
If $\scS\cap\scS_0\neq\emptyset$, choose an $i\in\scS\cap\scS_0$, 
then by Lemma \ref{lem:lines}, all $j$ such that $i$ and $j$ intersect at a double 
point must be contained in $\scS$. Thus, if $\scS\cap\scS_0\neq\emptyset$, 
$\scS\supset\scS_0$. By replacing $e_{\scS}$ by $e_{\scS}+w_2=
e_{\decID\smallsetminus\scS}$, 
we may assume $\scS\cap\scS_0=\emptyset$, in other words, 
$\scS\subset\scS_1$. 
Again by Lemma \ref{lem:lines}, $\scS$ must be either $\scS_1$ or $\emptyset$. 
This completes the proof. 
\end{proof}

\subsection{Milnor fiber of icosidodecahedral arrangement}
\label{sec:2toricosi}

\begin{theorem}
\label{milnortorsion}
Let $F_{\ID}$ be the Milnor fiber of the icosidodecahedral arrangement $\ID$. 
Then, 
\begin{enumerate}
\item 
$\rank_{\Z}H^1(F_{\ID}, \Z)=15$. 
\item 
$\rank_{\Z_2}H^1(F_{\ID}, \Z_2)\geq 16$. 
\item 
The integral first homology group $H_1(F_{\ID}, \Z)$ has $2$-torsion. 
\end{enumerate}
\end{theorem}
\begin{proof}
(1) Recall that $H_1(F_{\ID}, \Z)$ admits $\Z_{16}$ action by monodromy. 
Then the homology with complex coefficients 
has eigenspace decomposition 
\[
H^1(F_{\ID}, \C)=\bigoplus_{\lambda^{16}=1}H^1(F_{\ID}, \C)_{\lambda}. 
\]
Each eigenspace $H^1(F_{\ID}, \C)_{\lambda}$ is known to be isomorphic 
to $H^1(M(\decID), \scL_{\lambda})$ \cite{cs-mil}, where 
$\scL_{\lambda}$ is the complex rank one local system on 
$M(\decID)$ which has monodromy $\lambda$ along the meridian of 
each line $H\in\decID$ in $\C^2$. 
In particular, the $1$-eigen space is 
$H^1(F_{\ID}, \C)_1\simeq H^1(M(\decID), \C)\simeq \C^{15}$. 
For $\lambda\neq 1$, there are several practical 
ways to check that $H^1(F_{\ID}, \C)_{\lambda}=0$. 
One of the methods is to apply the result by Esnault-Schechtman-Viehweg 
\cite{esv} and Schechtman-Varchenko-Terao \cite{stv}. 
Let $a_1, \dots, a_{16}\in\C$. Assume the following three conditions. 
\begin{itemize}
\item[(P1)] 
$\exp\left(2\pi\sqrt{-1}\cdot a_i\right)=\lambda$, 
\item[(P2)] 
$\sum_{i=1}^{16}a_i=0$. 
\item[(P3)] 
Let $\overline{\ID}$ be the induced projective arrangement of $16$ lines. 
For each quadruple intersection $p\in\CP^2$ of $\overline{\ID}$, 
the sum $\sum_{H_i\ni p}a_i$ is not contained in $\Z_{>0}$. 
\end{itemize}
(In some literature, such a local system is called admissible \cite{na-ra}.) 
Then \cite{esv, stv} asserts that 
\begin{equation}
H^k(M(\decID), \scL_{\lambda})\simeq H^k(A_{\C}^\bullet(\decID), \eta\wedge -), 
\end{equation}
for $k\geq 0$, where $\eta=\sum_{i=1}^{15}a_ie_i\in\A_{\C}^1(\decID)$. 

Here let us illustrate how to prove $H^1(M(\decID), \scL_{\lambda})=0$ for 
$\lambda=-1$. We can choose $a_1, \dots, a_{16}$ as 
\[
\begin{split}
&a_1=a_2=\cdots=a_9=a_{10}=\frac{1}{2}, \\
&a_{11}=a_{12}=\cdots=a_{15}=-\frac{1}{2}\\
&a_{16}=-\frac{5}{2}. 
\end{split}
\]
Then at each quadruple point, the sum of $a_i$'s is either $0$ or $-2$. 
Thus the conditions (P1), (P2) and (P3) are verified. Set 
$\eta:=\sum_{i=1}^{15}a_ie_i\in A_\C^1(\decID)$. We can check 
$H^k(A_{\C}^\bullet(\decID), \eta\wedge -)=0$ by arguments similar to 
the proof of Proposition \ref{mod2aomoto}. 
The vanishing of the eigenspaces for other eigenvalues 
$\lambda$ can be proved in a similar way. 

We can obtain the same result also by using resonant band 
algorithms formulated in \cite{y-mil, y-res}. 

(2) follows from Theorem \ref{thm:2cover}, Proposition \ref{mod2aomoto} 
and Corollary \ref{cor:increase}. 

(3) By universal coefficient theorem and (1), if $H_1(F_{\ID}, \Z)$ does not have 
$2$-torsion, $H_1(F_{\ID}, \Z_2)\simeq H_1(F_{\ID}, \Z)\otimes \Z_2$ 
has rank $15$ over $\Z_2$. This contradicts (2). 
\end{proof}

\begin{remark}
The proof of Theorem \ref{milnortorsion} works more generally. 
Let $\A$ be an arrangement in $\C^3$. 
Suppose 
\begin{itemize}
\item 
$\#\A$ is a power of $2$. 
\item 
The first cohomology of the mod $2$ 
Aomoto complex of the deconing of $\A$ 
does not vanish. 
\item 
The first cohomology of the Milnor fiber $H^1(F_{\A}, \C)$ does not have 
nontrivial monodromy eigenspaces. 
\end{itemize}
Then we can conclude $H_1(F_{\A}, \Z)$ has $2$-torsion. 
\end{remark}

\begin{remark}
Proposition \ref{mod2aomoto} and Theorem \ref{milnortorsion} (1) 
show that $\ID$ is a counterexample to a conjectures in 
\cite[Conjecture 1.9]{PS-modular}. 
More precisely, the equality $\beta_2=e_2$ in \cite{PS-modular} does not hold for $\ID$. 
\end{remark}

\begin{remark}
Enrique Artal-Bartolo communicated to us that 
he checked by computer that 
$H_1(F_{\ID}, \Z)\simeq\Z^{15}\oplus\Z_2$. It would be a challenging problem 
to develop a method which can check the result theoretically. 
The following are also interesting problems. 
\begin{itemize}
\item[(1)] 
Describe the mod $2$ Betti numbers of the Milnor fibers of arrangements. 
\item[(2)] 
Describe the mod $2$ cohomology rings of double covers 
of arrangement complements. 
\end{itemize}
\end{remark}

\end{document}